\documentclass[12pt]{amsart}

\usepackage{amsmath, amsthm, amssymb}
\usepackage{url} % <-- Para p\'aginas web o similar: \url{...}
\usepackage{graphicx}
\usepackage{xcolor}
\usepackage{moreverb}

\usepackage{fancyvrb}

\def\r{\mathbb R}
\def\m{\mathbb M}
\def\s{\mathbb S}
\def\h{\mathbb H}

\usepackage{listings}
 \lstnewenvironment{code}[1][]
  {\lstset{#1}}% Add/update settings locally
  {}

\newtheorem{theorem}{Theorem}[section]
 \newtheorem{proposition}[theorem]{Proposition}
 
\theoremstyle{definition}

\newtheorem{example}[theorem]{Example}

\begin{document}

\title{Torqued and anti-torqued vector fields on hyperbolic spaces}

\author{Muhittin Evren Aydin$^1$}
\address{$^1$Department of Mathematics, Faculty of Science, Firat University, 23200, Elazig,   T\"urkiye}
\email{meaydin@firat.edu.tr}
\author{ Adela Mihai$^2$}
 \address{$^2$Technical University of Civil Engineering Bucharest,
Department of Mathematics and Computer Science, 020396, Bucharest, Romania
and Transilvania University of Bra\c{s}ov, Interdisciplinary Doctoral
School, 500036, Bra\c{s}ov, Romania}
 \email{adela.mihai@utcb.ro, adela.mihai@unitbv.ro}
\author{ Cihan Ozgur$^3$}
 \address{$^3$Department of Mathematics, Faculty of Arts and Sciences, Izmir Democracy University, 35140,  Izmir, T\"urkiye}
 \email{cihan.ozgur@idu.edu.tr}

\begin{abstract}
In this paper, we study the existence of torqued and anti-torqued vector fields on the hyperbolic ambient space $\h^n$. Although there are examples of proper torqued vector fields on open subsets of $\h^n$, we prove that there is not a proper torqued vector field globally defined on $\h^n$. Similarly, we have another non-existence result for anti-torqued vector fields  as long as their conformal scalar is a non-constant function satisfying certain conditions. When the conformal scalar is constant, some examples of anti-torqued vector fields are provided.
\end{abstract}
\maketitle

\section{Introduction} \label{intro}
%%%%%%%%%%%%%%%%%%%%%%%%%%%%%%%%

Let $(M,\langle, \rangle)$ be a Riemannian manifold, $\omega$ a $1$-form and $f$ a smooth function on $M$. Let also $\mathcal{V}$ be a nowhere zero smooth vector field on $M$. In 1944, Yano \cite{ya0} introduced a wide class of smooth vector fields known as {\it torse-forming} vector fields. A torse-forming vector field $\mathcal{V}$ is defined by the property that
\begin{equation}\label{tors}
\nabla^0_X \mathcal{V}=f X +  \omega(X) \mathcal{V}, \quad \forall X\in \mathfrak{X}(M),
\end{equation}
where $\nabla^0$ is the Levi-Civita connection on $(M,\langle, \rangle)$. We call the $1$-form $\omega$ and the function $f$ the {\it generating form} and {\it conformal scalar} (or {\it potential function}) of $\mathcal{V}$, respectively (see \cite{bo,mih0}). Denote by $\mathcal{W}$ the dual to $\omega$ on $M$, i.e. $\langle \mathcal{W}, X \rangle =\omega (X)$, for every $X \in \mathfrak{X}(M)$. We call $\mathcal{W}$ {\it generative} of $\mathcal{V}$.

An immediate class among such vector fields appear when $\omega$ is particularly assumed to be a zero map on $M$, which we call {\it concircular} vector fields \cite{s0,ya1}. This class has remarkable applications in Physics \cite{c0,dr,mam0,mam1,mam2,mam3}. Similarly, we have the class of {\it recurrent} vector fields when $f$ is identically zero \cite{cv}.

Another class of torse-forming vector fields is the {\it proper} ones, i.e., the generating form and conformal scalar are nowhere zero. More clearly, this condition means that there is not an open subset of $M$ such that $\omega$ and $f$ are identically zero. There are non-trivial examples of such vector fields in various ambient spaces. For every rotational hypersurface in the Euclidean $n$-dimensional space $\r^n$ with the axis passing through origin, the tangential part of the position vector field $\Phi$ is a proper torse-forming vector field \cite{c00}. Another non-trivial example is on the hypersphere $\s^{n-1}$ of $\r^n$; if $(x_1,...,x_n)$ are the canonical coordinates of $\r^n$, then $\mathcal{V}=e^{-x_1}(\partial_1  - x_1\Phi),$ with $\partial_1:=\partial/\partial x_1$, is a proper torse-forming vector field on $\s^{n-1}$ \cite{id}.

In this paper, we are interested in two particular types of proper torse-forming vector fields, called torqued and anti-torqued vector fields. In principle, these arise by emposing additional conditions on the generative $\mathcal{W}$. If $\mathcal{V}$ is a proper torse-forming vector field then these two types are defined based on whether $\mathcal{W}$ is perpendicular or parallel to $\mathcal{V}$. 

More explicitly, Chen \cite{crs3} defined a {\it torqued} vector field as a torse-forming vector field $\mathcal{V}$ with $\langle \mathcal{V}, \mathcal{W}\rangle =0$ on $M$. It was proved that a twisted product of an open interval and a Riemannian (or Lorentzian) manifold admits a torqued vector field (see \cite{crs3}). Such vector fields were used to characterize Ricci solitons \cite{c01}. On Einstein manifolds, every torqued vector field is proportional to a concircular vector field $\mathcal{T}$ by a smooth function $\lambda$ with $\mathcal{T}\lambda =0$ \cite{c01}. In addition, Deshmukh et al. proved that there is not a torqued vector field globally defined on the two standard simply-connected models of Riemannian space forms, the sphere $\s^n$ and the Euclidean space $\r^n$ \cite{dta}. In contrast, an example of such a vector field can be established on an open subset of $\r^n$ with nowhere vanishing conformal scalar \cite{dta}. Yoldas et al. \cite{yms} showed that a characteristic vector field of a Kenmotsu manifold cannot be torqued. Submanifolds of Kenmotsu manifolds were also characterized by using the torqued vector fields \cite{yms}.

The second type of proper torse-forming vector fields that we focus on is the {\it anti-torqued} vector fields introduced by Deshmukh et al. \cite{dan} (see also \cite{cs}). These are proper torse-forming vector fields with $\mathcal{W}=-f\mathcal{V}$ on $M$, which is contrary to the case of torqued vector fields with $\mathcal{W} \perp \mathcal{V}$ on $M$. Denote by $\nu$ the dual to $\mathcal{V}$. Then, by \eqref{tors}, an anti-torqued vector field $\mathcal{V}$ fulfills
\begin{equation}\label{a-tor}
\nabla^0_X \mathcal{V}=f (X - \nu (X)\mathcal{V}), \quad \forall X\in \mathfrak{X}(M).
\end{equation}
Since an anti-torqued vector field $\mathcal{V}$ is nowhere zero on $M$, its dual $\nu$ is also nowhere vanishing. Therefore, $\mathcal{V}$ is always proper as long as it is not parallel. In contrast, there are non-proper torqued vector fields that are not parallel, e.g. concircular vector fields. In addition, there are anti-torqued vector fields $\mathcal{V}$, where $\mathcal{V}=\nabla^0 g$, for some smooth function $g$ on $M$, which is not valid for the torqued vector fields. Anti-torqued vector fields are also known as {\it concurrent-recurrent} vector fields when $f$ is particularly assumed to be a nonzero constant \cite{na}. The necessary and sufficient condition for a Riemannian manifold to carry an anti-torqued vector field is that it is locally a warped product $I \times_\lambda F$, where $I$ is an open interval and $F$ a Riemannian manifold, as one can be seen in \cite[Theorem 3.1]{amc} and \cite[Theorem 3]{na}. 

Assuming the existence of a vector field nowhere zero on Riemannian manifolds governs their geometry and topology (see \cite{amc,dk,dan,dta,dmtv,id}). Motivated by this, first we study the problem of determining whether a torqued vector field is globally defined on the third standard model $\h^n$ of Riemannian space forms. The same problem is also considered for the anti-torqued case.

This paper is organized as follows. After giving the basic concepts in Section \ref{pre} relating to Riemannian (sub)manifolds, we provide non-trivial examples of torqued and anti-torqued vector fields in Section \ref{examp}. Moreover, we have a non-existence result for proper torqued vector fields on Riemannian manifolds as long as they are of constant length (Proposition \ref{torq-nonconst}). In contrast, there exist other torse-forming vector fields of constant length on Riemannian manifolds (Propositions \ref{anti-unit} and \ref{tors-cons}). In Section 4, we prove that a proper torqued vector field with a nowhere zero conformal scalar cannot be globally defined on $\h^n$ (Theorem \ref{torq-hyp-nonexis}). Similarly, when the conformal scalar is a non-constant function satisfying certain conditions, there is not an anti-torqued vector field globally defined on $\h^n$  (Theorem \ref{anti-hyp-nonexis}).

%%%%%%%%%%%%%%%%%%%%%%%%%%%%%%%%
\section{Preliminaries} \label{pre}
%%%%%%%%%%%%%%%%%%%%%%%%%%%%%%%%

Let $(M, \langle , \rangle) $ be a Riemannian manifold of dimension $n$, and let $\nabla^0$ the Levi-Civita connection of $M$. The Riemannian curvature tensor is defined by%
\begin{equation*}
R( X,Y) Z=\nabla^0_X\nabla^0_YZ-\nabla^0_Y\nabla^0_XZ-\nabla^0_{\left[ X,Y\right] }Z, \quad X,Y,Z \in \mathfrak{X}(M),
\end{equation*}
where $[ ,]$ is the bracket operation. Let $\Gamma $ be a plane section of $T_{p}M$ with a given basis $\{ x,y\}$. The sectional curvature $K( \Gamma ) $ of $\Gamma $ is defined by
\begin{equation*}
K( x,y) =\frac{\langle R( x,y) y,x \rangle }{\langle x,x \rangle \langle y,y\rangle -\langle x,y \rangle ^2}.
\end{equation*}

A Riemannian manifold $\m^n(c)$ is called a space form if $K$ is a constant $c$ for every plane section at every point. The standard simply-connected models of a Riemannian space form $\m^n(c)$ are the Euclidean space $\r^n=\m^n(0)$, the sphere $\s^n=\m^n(1)$ and the hyperbolic space $\h^n=\m^n(-1)$.

There are several models for the hyperbolic ambient space, two of them are utilized in this paper. Let $(x_1,...,x_n)$ be the canonical coordinates of $\r^n$. The first one is the upper half space model $\h^n=\{(x_1,...,x_n) \in \r^n : x_n>0\},$ which is endowed with the metric
\begin{equation*}
\langle ,\rangle =\frac{1}{x_n^2}\sum_{i=1}^{n}dx_i^2.
\end{equation*}

Let $\r_1^{n+1}$ be the Lorentz-Minkowski space endowed with the canocial Lorentzian metric    
\begin{equation*}
\langle ,\rangle =-dx_1^2+\sum_{i=2}^{n+1}dx_i^2.
\end{equation*}
Then, the second model that we concentrate on is the hyperboloid model 
\begin{equation*}
\h^n=\{P=(x_1,...,x_{n+1}) \in \r_1^{n+1} : \langle P,P \rangle=-1\}.
\end{equation*}

Let $N$ be a submanifold in a Riemannian manifold $M$. The Gauss formula is given by
\begin{equation*}
\nabla^0_XY=\nabla_XY+h(X,Y), \quad X,Y\in \mathfrak{X}(N),
\end{equation*}
where $\nabla$ is the induced connection on $N$ and $h$ is the second fundamental form. %Denote by $U $ a normal vector field on $N.$ %Then, the Weingarten formula is
%$$
%\nabla^0_XU=-A_U(X)+D_XU, \quad X\in \mathfrak{X}(N),
%$$
%where $A$ is the shape operator with $\langle h(X,Y),U \rangle =\langle A_U(X),Y \rangle$ and $D$ is the metric connection on the normal bundle.

%Let $R^N$ be the Riemannian curvature tensor of $N$. Then the Gauss equation is%
%\begin{eqnarray*}
%\langle  R^N( X,Y) Z,W \rangle &=&\langle R(X,Y) Z,W \rangle +\langle h( X,W) ,h(Y,Z) \rangle  \\
%&-&\langle h( X,Z) ,h( Y,W)\rangle, \quad X,Y,Z,W \in \mathfrak{X}(N).
%\end{eqnarray*}

%%%%%%%%%%%%%%%%%%%%%%%%%%%%%%%%
\section{Some Examples and Results on Torqued and Anti-Torqued Vector Fields} \label{examp}
%%%%%%%%%%%%%%%%%%%%%%%%%%%%%%%%

Let $\mathcal{V}$ be a torqued vector field on a Riemannian manifold $M$, i.e.
\begin{equation*}
\nabla^0_X\mathcal{V}=fX+\omega(X)\mathcal{V}, \quad \omega(\mathcal{V})=0,    
\end{equation*}
for every $X\in\mathfrak{X}(M)$. Assume that $M$ is a twisted product $I\times_\lambda F$, where $I$ is an open interval and $F$ a Riemannian manifold, and $\mu$ is a smooth function on $F$. From \cite{c01}, it is known that 
\begin{equation*}
\mathcal{V}=\lambda \mu \frac{\partial}{\partial s}, \quad s \in I,    
\end{equation*}
is always a torqued vector field tangent to $I$, where $s$ is the arc-length parameter of $I$. 

If the ambient space $M$ is a sphere or a Euclidean space, then such vector fields cannot be defined on $M$ globally \cite{dan,dta}. However, there are examples of proper torqued vector fields on open subsets of Euclidean space (see \cite[Example 1]{dta}) as well as of hyperbolic ambient spaces.

\begin{example}\label{ex-torq}
We consider the hyperboloid model of $\h^n$. If $\Phi: \h^n \to \r^{n+1}_1$ is the standard isometric immersion, then by the Gauss formula we write
\begin{equation*}
\nabla^0_XY=\nabla_XY + \langle X,Y \rangle \Phi, \quad X,Y \in \mathfrak{X}(\h^n),
\end{equation*}
where $\nabla^0$ stands for the Levi-Civita connection of $\r^{n+1}_1$ and $\nabla$ the induced one on $\h^n$. Set $\mathcal{T} =P_0+\langle P_0 , \Phi \rangle \Phi $, where $P_0 \neq 0$ is a fixed vector field in $\r^{n+1}_1$. From \cite{loy1}, it follows that
\begin{equation}
\nabla_X\mathcal{T} =\langle P_0 , \Phi  \rangle X, \quad \forall X \in \mathfrak{X}(\h^n).    \label{t-con}
\end{equation}
Next, we consider the function 
\begin{equation*}
g:\widetilde{M} \subset \h^n \to \r, \quad (x_1,...,x_{n+1}) \mapsto g(x_1,...,x_{n+1})= \frac{x_1}{x_{n+1}},    
\end{equation*}
where 
\begin{equation*}
\widetilde{M}= \{(x_1,...,x_{n+1}) \in \h^n : x_1x_2x_{n+1}\neq 0, x_1 \neq x_{n+1} \}.    
\end{equation*}
We set $P_0= \partial_2 $ and $ \mathcal{V}=e^g\mathcal{T}.$ From \eqref{t-con}, it follows
\begin{equation*}
\nabla_X\mathcal{V}=e^g\langle P_0 , \Phi \rangle X +X(g)\mathcal{V}, \quad \forall X \in \mathfrak{X}(\widetilde{M}).    
\end{equation*}
Defining $f:=e^g\langle P_0 , \Phi \rangle$ and $\omega:=dg$, the above expression becomes
\begin{equation*}
\nabla_X\mathcal{V}=f X +\omega(X)\mathcal{V}, \quad \forall X \in \mathfrak{X}(\widetilde{M}).    
\end{equation*}
Because $\mathcal{V}(g)=0$, we conclude $\omega(\mathcal{V})=0$, which implies that $\mathcal{V}$ is a torqued vector field on $\widetilde{M}$. Next, we show that $\mathcal{V}$ is a proper torqued vector field on $\widetilde{M}$. The conformal scalar of $\mathcal{V}$ is given by
\begin{equation*}
f(x_1,...,x_{n+1})=x_2 e^{x_1/x_{n+1}},
\end{equation*}
is clearly nowhere zero on $\widetilde{M}$. Moreover, the generating form of $\mathcal{V}$ is
\begin{equation*}
\omega=\frac{1}{x_{n+1}}dx_1-\frac{x_1}{x^2_{n+1}}dx_{n+1} \neq 0 \text{ on } \widetilde{M}.
\end{equation*}
Consequently, $\mathcal{V}$ is proper torqued vector field on $\widetilde{M}$. 
\end{example}

On any Riemannian manifold, we have a non-existence result for a proper torqued vector field when it is of constant length. See also \cite[Theorem 3]{ag}.

\begin{proposition} \label{torq-nonconst}
A proper torqued vector field on a Riemannian manifold is never of constant length.
\end{proposition}

\begin{proof}
Let $\mathcal{V}$ be a torqued vector field on a Riemannian manifold $M$, i.e.
\begin{equation*}
 \nabla^0_X\mathcal{V}=fX+\omega(X)\mathcal{V}, \quad \omega(\mathcal{V})=0,   
\end{equation*}
for every $X\in\mathfrak{X}(M)$. By contradiction, assume that
\begin{equation*}
0=X(|\mathcal{V}|)=f|\mathcal{V}|^{-1}\langle X,\mathcal{V} \rangle +|\mathcal{V}|\langle X,\mathcal{W} \rangle , \quad \forall X \in \mathfrak{X}(M),  
\end{equation*}
where $\mathcal{W}$ is generative of $\mathcal{V}$. Writing $X=\mathcal{V}$ and $X=\mathcal{W}$, we see that $f$ and $\omega$ are identically zero, which is not possible.
\end{proof}

Let $\mathcal{V}$ be an anti-torqued vector field on a Riemannian manifold $M$. If $M$ is a Euclidean space, then there are examples of $\mathcal{V}$ globally defined on $M$ but not on a sphere (see \cite{dan}). Next we give an example of $\mathcal{V}$ when $M$ is the hyperbolic space $\h^n$. 
\begin{example} \label{ex3.1}
In the upper-half space model of $\h^n$, an orthonormal basis $\{e_1,...,e_n \}$ can be chosen as
\begin{equation*}
e_j=x_n\partial _j, \quad e_n=-x_n\partial _n, \quad j=1,...,n-1.    
\end{equation*}
The Levi-Civita connection gives, for $i=1,...,n$ and $j=1,...,n-1,$ 
\begin{equation*}
\nabla^0 _{e_j}e_{n}=e_j, \quad \nabla^0 _{e_n}e_n=0, \quad \nabla^0
_{e_i}e_j=0 \text{ } (i\neq j), \quad \nabla^0 _{e_j}e_j=-e_n.   
\end{equation*}
It is direct to see that $e_n$ is an anti-torqued vector field, namely
\begin{equation*}
\nabla^0 _{X}e_n =X-\langle X,e_{n}\rangle e_n, \quad \forall X\in \mathfrak{X}(\h^n),    
\end{equation*}
where the conformal scalar is identically $1$.
\end{example}

Similar to Example \ref{ex-torq}, on open subsets of $\h^n$, we also may have other examples of anti-torqued vector fields.

\begin{example}\label{ex-antitorq}
As in Example \ref{ex-torq}, we consider the hyperboloid model of $\h^n$. We introduce 
\begin{equation*}
 g:\widetilde{M} \subset \h^n \subset \r_1^{n+1} \to \r, \quad (x_1,...,x_{n+1}) \mapsto g(x_1,...,x_{n+1})= \frac{1}{x_{n+1}},   
\end{equation*}
where
\begin{equation*}
 \widetilde{M}=\{(x_1,...,x_{n+1}) \in \h^n : x_{n+1}\neq 0 \}.   
\end{equation*}
We set $P_0=\partial_{n+1}$, 
$\mathcal{T}=P_0+\langle P_0 , \Phi \rangle \Phi=\partial_{n+1}+x_{n+1}\Phi$,  and $\mathcal{V}=g\mathcal{T},
$
where $\Phi$ is the position vector field of $\r_1^{n+1}$. We will show that $\mathcal{V}$ is an anti-torqued vector field on $\widetilde{M}$. Let $\nabla$ be the induced Levi-Civita connection on $\h^n$ from $\r_1^{n+1}$. Since we know that
\begin{equation*}
\nabla_X\mathcal{T}=\langle P_0 , \Phi \rangle X, \quad \forall X \in \mathfrak{X}(\h^n),    
\end{equation*}
the following expression occurs
\begin{equation*}
\nabla_X\mathcal{V}= \frac{1}{x_{n+1}}\langle P_0 , \Phi \rangle X +X\left( \frac{1}{x_{n+1}} \right )\mathcal{T}    , \quad \forall X \in \mathfrak{X}(\widetilde{M}), 
\end{equation*}
or equivalently
\begin{equation}
\nabla_X\mathcal{V}=  X +X\left( \frac{1}{x_{n+1}} \right )\mathcal{T}, \quad \forall X \in \mathfrak{X}(\widetilde{M}). \label{ex2-01}
\end{equation}
A direct calculation gives
\begin{equation*}
X\left( \frac{1}{x_{n+1}} \right )=- \frac{1}{x_{n+1}^2}\langle P_0, X\rangle.   
\end{equation*}
Since $\langle \Phi ,X \rangle =0$ for every $X\in \mathfrak{X}(\h^n)$, we can have 
\begin{equation*}
\langle P_0, X\rangle= \langle P_0 + \langle P_0 , \Phi \rangle \Phi , X \rangle =\langle  \mathcal{T},X \rangle, \quad \forall X \in \mathfrak{X}(\h^n). 
\end{equation*}
Writing in Equation \eqref{ex2-01}, we obtain
\begin{equation*}
\nabla_X\mathcal{V}=X-\left \langle \frac{1}{x_{n+1}} \mathcal{T}, X \right \rangle \left (\frac{1}{x_{n+1}}\mathcal{T} \right ), \quad X \in \mathfrak{X}(\widetilde{M})    
\end{equation*}
or equivalently
\begin{equation*}
\nabla_X\mathcal{V}=X-\left \langle \mathcal{V}, X \right \rangle \mathcal{V}, \quad \forall X \in \mathfrak{X}(\widetilde{M}),  
\end{equation*}
which implies that $\mathcal{V}$ is an anti-torqued vector field on $\widetilde{M} \subset \h^n$ with constant conformal scalar $1$.
\end{example}

In contrast to Proposition \ref{torq-nonconst}, there are examples of anti-torqued vector fields of constant length on a Riemannian manifold $M$ as in Example \ref{ex3.1}. Moreover, if an anti-torqued vector field $\mathcal{V}$ on $M$ is of constant length, then it must be a unit geodesic vector field. Recall that a smooth vector field $\mathcal{T}$ on a Riemannian manifold is called a {\it unit geodesic} vector field if $\nabla^0_{\mathcal{T}}\mathcal{T}=0$. Hence, if $\mathcal{V}$ is a non-parallel anti-torqued vector field  whose length is a nonzero constant, then we have 
\begin{equation*}
0=X(|\mathcal{V}|)=f|\mathcal{V}|^{-1}(1-|\mathcal{V}|^2)\langle X,\mathcal{V} \rangle , \quad \forall X \in \mathfrak{X}(M),   
\end{equation*}
which implies $|\mathcal{V}|=1$. It then follows from \eqref{a-tor} that $\nabla^0_{\mathcal{V}}\mathcal{V}=0$.

Therefore, we have proved
\begin{proposition}  \label{anti-unit}
An anti-torqued vector field of constant length on a Riemannian manifold is always a unit geodesic vector field. 
\end{proposition}

Notice that Proposition \ref{anti-unit} is not valid for any proper torqued vector field $\mathcal{V}$ due to Proposition \ref{torq-nonconst}. Moreover, otherwise it follows from \eqref{tors} that
\begin{equation*}
0=\nabla^0_{\mathcal{V}}\mathcal{V}=f\mathcal{V}+\omega(\mathcal{V})\mathcal{V},    
\end{equation*}
or equivalently $f\mathcal{V}=0$, which is not possible. On the contrary, every proper torqued vector field is always a non-trivial geodesic vector field whose potential function is the conformal scalar $f$ of $\mathcal{V}$ (see  \cite{dmtv}).

A final conclusion regarding the  torse-forming vector fields of constant length is the following.

%Let $\mathcal{V}$ be a torse-forming vector field of constant length $|\mathcal{V}|>0$ given by \eqref{tors} and $\mu$ its dual. Also, let $\mathcal{W}$ be the genrating vector field, i.e. the dual of $\omega$. Because $\langle \nabla^0_X\mathcal{V} , \mathcal{V} \rangle =0$, we have
%$$
%0=f \langle X, \mathcal{V} \rangle +|\mathcal{V}|^{2}\omega(X)
%$$
%or equivalently $\omega=-|\mathcal{V}|^{-2}f  \mu.$ This means that $\mathcal{W}$ is parallel to $\mathcal{V}$. If in addition $|\mathcal{V}|=1$, then $\mathcal{V}$ is a unit anti-torqued vector field. Se, we conclude that every unit-torse-forming vector field is an anti-torqued vector field.

\begin{proposition}\label{tors-cons}
If a proper torse-forming vector field $\mathcal{V}$ on a Riemannian manifold is of constant length, then it is parallel to the generative. Particularly, every unit torse-forming vector field is an anti-torqued vector field. 
\end{proposition}
\begin{proof}
Let $\mathcal{V}$ be a torse-forming vector field on a Riemannian manifold $M$ of constant length $|\mathcal{V}|>0$, and let $\mu$ be the dual of $\mathcal{V}$.  Using \eqref{tors}, we have
\begin{equation*}
 0=\langle \nabla^0_X\mathcal{V} , \mathcal{V} \rangle =f \langle X, \mathcal{V} \rangle +|\mathcal{V}|^{2}\omega(X),   \quad \forall X\in \mathfrak{X}(M),
\end{equation*}
or equivalently $\omega=-|\mathcal{V}|^{-2}f  \mu.$ This means that $\mathcal{V}$ is parallel to its generative. Moreover, if $|\mathcal{V}|=1$, then it is an anti-torqued vector field.
\end{proof}

%%%%%%%%%%%%%%%%%%%%%%%%%%%%%%%%%%%
\section{Non-Existence Results on (Anti-)Torqued Vector Fields on $\h^n$} \label{char}
%%%%%%%%%%%%%%%%%%%%%%%%%%%%%%%%%%%

From \cite{dta}, we know that there are no proper torqued vector fields globally defined on either a sphere or on a Euclidean ambient space. In this section, we establish a similar non-existence result for those vector fields on the other standard complete simply-connected model of Riemannian space forms, more precisely on hyperbolic spaces.

\begin{theorem} \label{torq-hyp-nonexis}
There is not a proper torqued vector field globally defined on $\h^n$ with a conformal scalar $f$ that is nowhere zero.
\end{theorem}

\begin{proof}
The proof is by contradiction. Suppose that $\mathcal{V}$ is a proper torqued vector field globally defined on $\h^n$. Then, we have
\begin{equation} \label{torq-nonex}
\nabla^0_X \mathcal{V}=f X +  \omega(X) \mathcal{V}, \quad \omega(\mathcal{V})=0,  \quad \forall X\in \mathfrak{X}(M).
\end{equation}
By assumption, the conformal scalar $f$ is nowhere zero on $\h^n$. Denote by $\mathcal{W}$ the generative of $\mathcal{V}$ which is dual to the generating form $\omega$. Since $\omega (\mathcal{V})=0$, we conclude that $\mathcal{V} \perp \mathcal{W}$. The curvature tensor of $\h^n$ is%
\begin{equation} \label{torq-hyp-curv1}
R(X,Y)\mathcal{V}=\langle X,\mathcal{V}\rangle Y-\langle Y,%
\mathcal{V}\rangle X, \quad \forall X,Y\in \mathfrak{X}(\h^n)  .
\end{equation}%
From \eqref{torq-nonex}, it follows
\begin{eqnarray}  \label{torq-hyp-curv2}
R(X,Y)V =\nabla^0 _{X}( fY+\omega (Y)\mathcal{V}) -\nabla^0 _{Y}(
fX+\omega (X)\mathcal{V}) -  f[X,Y]-\omega ( [ X,Y]) \mathcal{V}  \nonumber \\
=\{ -Y(f)+f\omega (Y)\} X+\{ X(f)-f\omega (X)\} Y 
+d\omega (X,Y)\mathcal{V}.   \nonumber
\end{eqnarray}%
By \eqref{torq-hyp-curv1} and the above equation, we have%
\begin{equation} \label{torq-hyp-curv3}
\langle X,\mathcal{V}\rangle Y-\langle Y,\mathcal{V}%
\rangle X =\{ -Y(f)+f\omega (Y)\} X+\{ X(f)-f\omega (X)\}
Y+d\omega (X,Y)\mathcal{V}.  
\end{equation}%
Since this identity holds for every $X,Y\in \mathfrak{X}(\mathfrak{\mathbb{H}%
}^{n}),$ so does a linearly independent triple $\{ X,Y,\mathcal{V}%
\} $. Then, it must be $d\omega =0$. By \eqref{torq-hyp-curv3}, we conclude%
\begin{equation*}
\langle X,\mathcal{V}\rangle =X(f)-f\omega (X), \quad \forall X\in \mathfrak{X}(\h^n),    
\end{equation*}
which implies%
\begin{equation*}
\langle X,\mathcal{V}-\nabla f+f\mathcal{W}\rangle =0, \quad \forall X\in \mathfrak{X}(\h^n),    
\end{equation*}
or equivalently
\begin{equation*}
\nabla f=\mathcal{V}+f\mathcal{W}.    
\end{equation*}

Suppose that $\nabla f(p)=0,$ for some point $p\in \mathfrak{X}(\h^n).$ Since $f$ is nowhere zero, it follows that $f(p)\neq 0$, and hence $\mathcal{V}$ and $\mathcal{W}$ would be linearly dependent at $p,$ which contradicts the fact that $\mathcal{V} \perp \mathcal{W}$. Consequently, we are able to suppose that $\nabla f(p)\neq 0,$ for every point $p\in \h^n.$ Then, $f:\h^n\to \r$ is a submersion and every level set $\Sigma_x=f^{-1}\{ f(p)\} $ is
a compact submanifold of dimension $n-1$ (see \cite{lee}). Introduce a vector field $\mathcal{T}\in \mathfrak{X}(\h^n)$
\begin{equation*}
\mathcal{T}(p)=\frac{\nabla f(p)}{\langle \nabla f(p),\nabla f(p)\rangle }, \quad p\in \h^n.   
\end{equation*}
It is direct to observe that $\mathcal{T}(f)=1$ and so the local
one-parameter group of local transformations $\{ \varphi _{t}\}$ of $\mathcal{T}$ satisfies%
\begin{equation} \label{local}
\mathcal{T}( \varphi_{t}(p)) =f(p)+t, \quad t \in \r.
\end{equation}%
Using escape Lemma (see \cite{lee}) and expression \eqref{local}, we understand
that $\mathcal{T}$ is a complete vector field and $\{ \varphi_{t}\} $ is one-parameter group of transformations of $\h^n.$ We next define a smooth function 
\begin{equation*}
g:\r \times \Sigma _{x}\to \h^n, \quad g(t,q)=\varphi _{t}(q).    
\end{equation*}
Notice that 
\begin{equation*}
\varphi _{t_{1}}\circ \varphi _{t_{2}}=\varphi _{t_1+t_2}, \quad t_1,t_2\in \r.    
\end{equation*}
Then, for every $q\in \h^n,$ we can find a parameter $t\in \r$ and a point $\varphi_t(q)=q'\in \Sigma_x$ with $q=\varphi_{-t}(q'),$ yielding that $%
g(-t,q')=q.$ Also, if $g(t_1,q_1)=g(t_2,q_2),$ then 
\begin{equation} \label{local2}
\varphi _{t_1}(q_1)=\varphi _{t_2}(q_2) 
\end{equation}
and expression \eqref{local} gives 
\begin{equation*}
f(q_1)+t_1=f(q_2)+t_2.    
\end{equation*}
Due to $q_1,q_2\in \Sigma_x$ and $t_1,t_2\in \r$, we see
that $f(q_1)=f(q_2)$ and $t_1=t_2,$ namely from \eqref{local2} it follows $q_1=q_2.$ Hence $g$ is a one-to-one and onto mapping with its inverse%
\begin{equation*}
g^{-1}(q)=(-t,q')=(-t,\varphi_t(q)).   \end{equation*}

Consequently, we observe that $\r\times \Sigma_x$ is
diffeomorphic to $\h^n$, implying $\Sigma_x$ is diffeomorphic
to $\h^{n-1}$. This is not possible because $\Sigma_x$ is a compact subset of $\h^n$. 
\end{proof}

Although there is no an anti-torqued vector field globally defined on a sphere (see \cite{dan}), there are examples in Euclidean and hyperbolic ambient spaces, as mentioned in Section \ref{examp}. Moreover, in $3$-dimensional setting, if the metric of a Riemannian manifold admitting an anti-torqued vector field with constant conformal scalar is Ricci soliton (or gradient Ricci almost soliton), then it is of constant negative curvature \cite{na}. However, as long as the conformal scalar is a non-constant function satisfying certain conditions, we obtain a non-existence result for globally defined anti-torqued vector fields on $\h^n$.

\begin{theorem}  \label{anti-hyp-nonexis}
There is not an anti-torqued vector field globally defined on $\h^n$ whose conformal scalar is non-constant and nowhere takes values $-1,0$, or $1$.
\end{theorem}

\begin{proof}
By contradiction suppose that $\mathcal{V}$ is an anti-torqued vector field on $\h^n$. Using \eqref{a-tor}, we have
\begin{eqnarray*} \label{anti-hyp}
R(X,Y)\mathcal{V}=-\{ Y(f)+f^{2}\nu (Y)\} X+\{ X(f)+f^{2}\nu(X)\} Y  +\{ -X(f)\nu (Y)+Y(f)\nu (X)+d\nu (X,Y)\} \mathcal{V}. 
\end{eqnarray*}%
By \eqref{torq-hyp-curv1} and the above equation, we have%
\begin{eqnarray*}  \label{anti-hyp2}
\langle X,\mathcal{V}\rangle Y-\langle Y,\mathcal{V}\rangle X =-\{ Y(f)+f^2\nu (Y)\} X+\{
X(f)+f^2\nu (X)\} Y   \\
+\{ -X(f)\nu (Y)+Y(f)\nu (X)+d\nu (X,Y)\} \mathcal{V}. 
\nonumber
\end{eqnarray*}%
Assume that $\{ X,Y,\mathcal{V}\} $ is a linearly independent set. Then, using $\nu (X)=\langle X,\mathcal{V}\rangle ,$  
\begin{eqnarray*}
X(f)+\left( f^{2}-1\right) \nu (X) =0, \\
Y(f)+\left( f^{2}-1\right) \nu (Y) =0, \\
-X(f)\nu (Y)+Y(f)\nu (X)+d\nu (X,Y) =0,
\end{eqnarray*}%
for every $X,Y\in \mathfrak{X}(\h^n).$ Obviously, we
have $d\nu =0$ and
\begin{equation*}
\nabla f=(1-f^{2})\mathcal{V}. \end{equation*}

By the hypothesis, $f$ is non-constant and so we may assume $\nabla f(p)\neq 0$ for every $p \in \h^n$ such that $f(p) \neq \pm 1$. The proof can be completed by following the same way as in that of Theorem. \ref{torq-hyp-nonexis}.
\end{proof}

%%%%%%%%%%%%%%%%%%%
\section*{Acknowledgements}
%%%%%%%%%%%%%%%%%%%
The authors express their sincere gratitude to Professor Sharief Deshmukh for his valuable and insightful discussions.

%\section{The second section}

%%% ENTER REFERENCES IN THE FORM

\end{document}